\documentclass[a4paper,fleqn]{article}

\usepackage{fullpage}
\usepackage{amsmath}
\usepackage{cite}
\usepackage{tikz}
\usepackage{tikzscale}
\usetikzlibrary{matrix}
\newlength\figureheight
\newlength\figurewidth
\usepackage{pgfplots}
\usepackage{amssymb}
\usepackage{mathdots}
\usepackage{multirow}
\usepackage{algorithm}
\usepackage{algpseudocode}

\usepackage{subcaption}

\usepackage{url}

\usepackage{amsthm}
\newtheorem{theorem}{Theorem}[section]
\newtheorem{lemma}[theorem]{Lemma}

\title{A $\star$-product solver with spectral accuracy for non-autonomous ordinary differential equations}
\providecommand{\keywords}[1]{\textbf{\hspace{0.9cm}Keywords: } #1}
\date{}
\author{ Stefano Pozza \footnotemark[2]	\and Niel Van Buggenhout\footnotemark[2]}

\begin{document}

	\maketitle
		\renewcommand{\thefootnote}{\fnsymbol{footnote}}
	
	\footnotetext[2]{Charles University, Sokolovská 83 186, 75 Praha 8, Czech Republic. (pozza@karlin.mff.cuni.cz, buggenhout@karlin.mff.cuni.cz )}
	\footnotetext{This work was supported by Charles University Research programs No. PRIMUS/21/SCI/009 and UNCE/SCI/023, and by the Magica project ANR-20-CE29-0007 funded by the French National Research Agency.}

	\begin{abstract}
		A new method for solving non-autonomous ordinary differential equations is proposed, the method achieves spectral accuracy.
		It is based on a new result which expresses the solution of such ODEs as an element in the so called $\star$-algebra.
		This algebra is equipped with a product, the $\star$-product, which is the integral over the usual product of two bivariate distributions.
		Expanding the bivariate distributions in bases of Legendre polynomials leads to a discretization of the $\star$-product and this allows for the solution to be approximated by a vector that is obtained by solving a linear system of equations.
		The effectiveness of this approach is illustrated with numerical experiments.
	\end{abstract}

	\section{Introduction}
	A new method for solving non-autonomous ordinary differential equations that achieves spectral accuracy \cite[Chapter 21]{Tr13} is proposed.
	Consider a smooth function $\tilde{f}\in C^{\infty}$ that is analytic, then the ODE studied here is
	\begin{equation}\label{eq:ODE}
		\frac{d}{dt} \tilde{u}(t) = \tilde{f}(t) \tilde{u}(t), \quad \tilde{u}(-1) = 1, \quad t\in\mathcal{I}, \quad \mathcal{I}:=\left[-1,1\right].
	\end{equation}
	This method forms an essential building block in the development of a numerical method to solve the matrix ODE
	\begin{equation*}
		\frac{d}{dt} \tilde{U}(t) = \tilde{A}(t) \tilde{U}(t), \quad \tilde{U}(-1) = I, \quad t\in\mathcal{I},
	\end{equation*}
	for sparse large-to-huge matrix functions $\tilde{A}(t)$.
	The matrix ODE appears in many applications, e.g., in nuclear magnetic resonance spectroscopy (NMR) \cite{HaSp98}.
	In NMR $\tilde{A}(t) = - 2\pi \imath H(t)$, where $H(t)$ is the Hamiltonian of the system describing the dynamics of the nuclear spins of some sample in a magnetic field. 
	Hamiltonians appearing in NMR are of size $2^\ell \times 2^\ell$ for a system with $\ell$ spins and is usually sparse since spins only interact with close neighbors.\\
	The new method for the scalar ODE \eqref{eq:ODE} is based on expressing this ODE and its solution $\tilde{u}(t)$ in the so called $\star$-algebra, which is equipped with the $\star$-product \cite{GiPo20}.
	Section \ref{sec:starAlgebra} introduces the $\star$-product, which is an integral over two bivariate distributions, and the $\star$-algebra generated by this product.
	In this algebra, the solution is given by a closed form expression \cite{GiPo21}.
	In principle, this expression can be computed symbolically, however, in general, the computation is too complex for practical purposes.\\
	Therefore, a numerical procedure is proposed that computes a discretization of this solution in the matrix algebra, equipped with the usual matrix-matrix product.
	The key to going from the $\star$-algebra to the matrix algebra is finding a suitable discretization of the $\star$-product, which can be based on a quadrature rule \cite{CiPoRZVB22} or on the expansion of the bivariate distributions in a basis of orthonormal polynomials.
	The latter approach is followed in this paper, a basis of orthonormal Legendre polynomials is chosen and the resulting discretization is discussed in Section \ref{sec:discr}.
	Section \ref{sec:approx} describes how to obtain an accurate approximation inside the matrix algebra and illustrates the method with some numerical experiments.
	
	\section{$\star$-product}\label{sec:starAlgebra}
	The solution $\tilde{u}(t)$ can be expressed as a closed form expression in the $\star$-algebra.
	First, all functions and the ODE \eqref{eq:ODE} are represented in the space
	$\mathcal{D}(\mathcal{I})$, which is spanned by all distributions $d$ of the form $d(t,s) = \tilde{d}(t,s)\Theta(t-s) + \sum_{i=0}^N \tilde{d}_i(t,s) \delta^{(i)}(t-s)$, with $N\geq 0$ and $\tilde{d},\tilde{d}_i$ are smooth bivariate functions over $\mathcal{I}\times \mathcal{I}$.
	In $\mathcal{D}(\mathcal{I})$ the ODE becomes
	\begin{equation}\label{eq:ODE_bivariate}
		\frac{d}{dt} u(t,s) = f(t,s) u(t,s),\quad u(s,s) = 1,\quad t,s\in\left[-1,1\right].
	\end{equation}
	The function is $f(t,s)=\tilde{f}(t)\Theta(t-s)$, where the Heaviside function $\Theta(t-s) := \begin{cases}
		1, \quad \text{if }t\geq s,\\
		0, \quad \text{if } t<s\end{cases}$ enforces the starting time, which is given by the parameter $s$.
	The space of smooth functions $\tilde{f}(t)$ multiplied with the Heaviside function $\Theta(t-s)$ is denoted by $C^\infty_{\Theta} := \{f: f(t,s) = \tilde{f}(t)\Theta(t-s), \quad \tilde{f}\in C^\infty\}$.
	For $f,g\in C^{\infty}_{\Theta}$ the $\star$-product, defined as
	\begin{equation}\label{eq:starProd}
		f(t,s)\star g(t,s) := \int_{-1}^1 g(t,\tau) f(\tau,t) d\tau \in C^{\infty}_{\Theta},
	\end{equation}
	is closed.
	In the larger space $\mathcal{D}(\left[-1,1\right])$ the $\star$-product is also closed and an inverse for this product, $f(t,s)^{-\star}$ such that $f(t,s)\star f(t,s)^{-\star} = \delta(t-s)$, exists under certain conditions on $f(t,s)$ \cite{GiPo20}.
	The corresponding identity element is the Dirac impulse $\delta(t-s) = \begin{cases}
		1, \quad \text{if }t= s,\\
		0, \quad \text{else}
	\end{cases}.$
	The elements and operations which compose the $\star$-algebra are given in Table \ref{table:starAlgebra}, for details we refer to \cite{GiPo20}.
	\begin{table}[!ht]
		\centering
		\begin{tabular}{l|l}
			Operation/element & Properties\\				\hline
			$f(t,s) \star g(t,s) = \int_{-1}^1 g(t,\tau) f(\tau,t) d\tau$             &  $\star$-product is closed in $\mathcal{D}(\mathcal{I})$ \\
			$f + g$           & addition is closed in $\mathcal{D}(\mathcal{I})$ \\
			$1_{\star} = \delta(t-s)$ (Diract delta) &  identity element for the $\star$-product      \\
			$f^{\star-1}(t,s)$                    & $\star$-inverse \cite{GiPo20}    \\
			$R_\star(f)(t,s):=(1_{\star}- f)^{\star-1}(t,s)$   &  $\star$-resolvent \cite{GiPo20}                                      
		\end{tabular}
		\caption{Summary of elements and operations which compose the $\star$-algebra. The functions $f,g,q\in \mathcal{D}(\left[-1,1\right])$.}
		\label{table:starAlgebra}
	\end{table}
	
	In the $\star$-algebra the solution $\tilde{u}(t)$ to \eqref{eq:ODE} is given by evaluating the solution $u(t,s)$ in $s=-1$:
	\begin{equation*}
		\tilde{u}(t) = u(t,s)\vert_{s=-1}, \quad \text{with }u(t,s) = \Theta(t-s)\star R_{\star}(\tilde{f}(t)\Theta(t-s)).
	\end{equation*}
	Computing $u(t,s)$ symbolically is usually too complex, therefore we will discretize the problem and compute an approximation to the solution by numerical computation.

	\section{From $\star$-algebra to matrix algebra}\label{sec:discr}
	The key to going from the $\star$-algebra to the matrix algebra is discretizing the $\star$-product.
	One way to discretize the $\star$-product is by the use of quadrature rules \cite{CiPoRZVB22}.
	Another way is by expansion in a basis of orthonormal polynomials (ONPs), which is the topic of this paper.
	A natural choice of ONPs is the sequence of Legendre polynomials, these polynomials are discussed in Section \ref{sec:LegPoly}.
	Using these polynomials as a basis, Section \ref{sec:expansion} describes how distributions living in $C^{\infty}_\Theta$ can be expanded as a series and Section \ref{sec:coeffs} provides details on how to compute the coefficients in this series.
	The discretization of the $\star$-product that follows from the Legendre basis expansion and the resulting matrix algebra are the topic of Section \ref{sec:finBasis}.
	
	\subsection{Legendre polynomials}\label{sec:LegPoly}
	The sequence of Legendre polynomials $\{p_k\}_k$ satisfies the orthogonality conditions
	\begin{equation*}
		\int_{-1}^1 p_k(x) p_{\ell}(x) dx \begin{cases}
			=0,\quad \text{if } k\neq \ell\\
			\neq 0,\quad \text{if } k=\ell
		\end{cases}.
	\end{equation*}
	We choose to normalize this sequence such that $\int_{-1}^1 p^2_k(x)dx = 1$.
	These orthonormal Legendre polynomials satisfy the property stated in Lemma \ref{prop:int_shift_Legendre}, which is paramount to efficiently computing the Legendre series expansion of funtions in $C^\infty_\Theta$.
	\begin{lemma}\label{prop:int_shift_Legendre}
		Consider orthonormal Legendre polynomials $\{p_\ell(t)\}_{l}$.
		Then for $-1\leq\tau\leq 1$ the following equality holds, for $\ell>0$,
		\begin{equation*}
			\int_{-1}^{\tau} p_\ell(\rho)d\rho = \frac{1}{\sqrt{2\ell+1}} \left(\frac{1}{\sqrt{2\ell+3}}p_{\ell+1}(\tau) - \frac{1}{\sqrt{2\ell-1}} p_{\ell-1}(\tau)  \right)
		\end{equation*}
		and for $\ell=0 $
		\begin{equation*}
			\int_{-1}^{\tau} p_0(\rho)d\rho = \frac{1}{\sqrt{3}} p_1(\tau) + p_0(\tau).
		\end{equation*}
	\end{lemma}
	The expansion of a given function $\tilde{f}(x)$ in a basis of Legendre polynomials is given by the series
	\begin{equation*}
		\tilde{f}(x) = \sum_{k=0}^{\infty} \tilde{f}_k p_k(x), \quad \text{with coefficients }\tilde{f}_k = \int_{-1}^1 \tilde{f}(\tau) p_k(\tau) d\tau.
	\end{equation*}
	For functions finite and continuous on $\mathcal{I}$ the Legendre series expansion is uniformly convergent \cite{WaXi12}.
	All functions we encountered in our applications are entire functions, i.e., analytic on the whole complex plane.
	The Legendre series expansions of entire functions converges faster than geometric \cite{Tr13}.
	For more details on the rate of convergence for analytic and differentiable functions see \cite{WaXi12}.\\
	Consider the truncated Legendre series $\hat{f}_N(x) := \sum_{k=0}^N \tilde{f}_k p_k(x)$.
	An upper bound for the error of $\hat{f}_N(x)$ to $\tilde{f}(x)$ can be obtained by noting that $\vert p_k(x)\vert \leq \sqrt{\frac{2k+1}{2}}$ on $\mathcal{I}$:
	\begin{equation*}
		\Vert \tilde{f}(x) - \hat{f}_N(x)  \Vert_\infty = \sum_{k=N+1}^\infty \hat{f}_k p_k(x) \leq \sum_{k=N+1}^\infty \vert \hat{f}_k\vert \sqrt{\frac{2k+1}{2}}.
	\end{equation*}
	The fast decay of the magnitude of the coefficients of smooth functions cancels out the square root growth as $k$ increases.
	Thus, if the series $\hat{f}_N(x)$ contains all the coefficients above machine precision appearing in the Legendre series expansion of $\tilde{f}(x)$, it represents $\tilde{f}(x)$ up to high accuracy.
	To be able to use the FFT, we will consider interpolating Legendre series instead of truncated Legendre series, the coefficients $\{f_k\}_{k=0}^{N-1}$ of the interpolating Legendre series can be computed using \verb|chebfun| \cite{DrHaTr14} at a complexity of $\mathcal{O}(N\log^2 (N))$.
	The accuracy of the interpolating Legendre series is expected to be close to the truncated Legendre series \cite[Chapter 4]{Tr13}, i.e., $\vert\tilde{f}_k-f_k\vert$ is small for $k=0,1,\dots N-1$.
	
	\subsection{Expansion of distributions living in $C^{\infty}_{\Theta}$}\label{sec:expansion}
	The expansion of $\tilde{f}(t) \Theta(t-s) = f(t,s)\in C^{\infty}_{\Theta}$ in Legendre bases is given by
	\begin{equation*}
		f(t,s) = \sum_{k=0}^{\infty} \sum_{\ell=0}^{\infty} f_{k,\ell} p_k(t) p_\ell(s), \quad \text{for } t\neq s, \quad \text{with coefficients } f_{k,\ell} = \int_{-1}^1 \int_{-1}^1 f(\tau,\rho) p_k(\tau) p_\ell(\rho) d\tau d\rho.
	\end{equation*}
	The coefficients form the coefficient matrix $F :=\left[f_{k,\ell}\right]_{k,\ell=1}^\infty$ which represents $f(t,s)$ in the bases of Legendre polynomials:
	\begin{equation*}
		f(t,s) \approx \begin{bmatrix}
			p_0(t)& p_1(t) &p_2(t) &\dots 
		\end{bmatrix} F \begin{bmatrix}
			p_0(s)\\
			p_1(s)\\
			p_2(s)\\
			\vdots
		\end{bmatrix}.
	\end{equation*}
	In the sequel, we will work with a truncation of this double series, we consider the $M\times M$ leading principal submatrix $F_M$ of $F$, such that $f(t,s)$ is represented by
	\begin{equation*}
		f(t,s) \approx \sum_{k=0}^{M-1} \sum_{\ell=0}^{M-1} f_{k,\ell} p_k(t) p_\ell(s)= \begin{bmatrix}
			p_0(t)& p_1(t)& p_2(t)& \dots & p_{M-1}(t)
		\end{bmatrix} F_M \begin{bmatrix}
			p_0(s)\\
			p_1(s)\\
			p_2(s)\\
			\vdots\\
			p_{M-1}(s)
		\end{bmatrix}.
	\end{equation*}
	Symbolic computation of the coefficients $f_{k,\ell}$ can be slow.
	In general, a numerical computation of the coefficients is needed, this can be achieved by, e.g., using a Gauss-Legendre quadrature rule to discretize the double integral.
	The number of nodes required to achieve accuracy close to machine precision depends on the given function $\tilde{f}(t)$.
	However, there is a more straightforward and more efficient approach, which is the topic of next section.

	\subsection{Computing basis coefficient matrices}\label{sec:coeffs}
	The coefficient matrix $F_M$ can be computed up to high precision.
	The procedure to compute $F_M$ requires as input only the function $\tilde{f}(t)$, size of the matrix $M$ and the chosen accuracy for the entries of $F_M$.
	The coefficients are obtained in two steps.
	First $\tilde{f}(t)\in C^{\infty}$ is represented, up to machine precision, by its interpolating Legendre series $\tilde{f}(t) \approx \sum_{k=0}^{N-1} {f}_k p_k(x)$.
	This series is obtained by using \verb|chebfun| \cite{DrHaTr14}, which automatically chooses the number of terms $N$.
	Second, since the series can be expressed as $f(t,s) \approx \sum_{d=0}^{N-1} \hat{f}_d \sum_{k=0}^{\infty}  \sum_{\ell=0}^{\infty} p_d(t) \Theta(t-s) p_k(t) p_\ell(s)$, it suffices to compute the coefficients for Legendre polynomials of degree $d$ in $C^\infty_{\Theta}$, i.e., $p_d(t)\Theta(t-s)$:
	\begin{equation*}
		b_{k,\ell}^{(d)} := \int_{-1}^1 \int_{-1}^1 p_d(\tau) \Theta(\tau-\rho) p_k(\tau) p_\ell(\rho) d\rho d\tau.
	\end{equation*}
	The corresponding basis coefficient matrices $\{B^{(d)}\}_{d=0}^{N-1}$ are infinite matrices $B^{(d)} := \left[b_{k,\ell}^{(d)} \right]_{k,l=0}^\infty$.
	Note that these basis coefficient matrices do not change for different $f(t,s)$, we only require the expansion coefficients $\{{f}_d\}_{d=0}^{N-1}$ of $\tilde{f}(t)$ in a basis of orthonormal Legendre polynomials in order to compute $F\approx F^{(N)} = \sum_{d=0}^{N-1} \hat{f}_d B^{(d)}$, where superscript $(N)$ denotes the number of terms used in the Legendre expansion of $\tilde{f}(t)$.
	This approximation is convenient since, as is shown in the following, it is possible to express $b_{k,\ell}^{(d)}$ analytically.
	\begin{lemma}[Integral of the product of three Legendre polynomials \cite{GiJeZe88}]\label{prop:intLeg}
		Consider integers $a,b,c\geq 0$ and set $s:=\frac{a+b+c}{2}$.
		The integral of the product of three orthonormal Legendre polynomials is
		\begin{equation*}
			\mathcal{P}_{a,b,c} := \int_{-1}^{1} {p}_a(\rho) {p}_b(\rho) {p}_c(\rho) d\rho 
			= 
			\begin{cases}
				0\quad \textnormal{if } a+b+c \textnormal{ odd or } s<\max(a,b,c), \textnormal{ or }a<\vert b-c \vert,\\ \frac{2(2a+1)(2b+1)(2c+1)}{\sqrt{8}(a+b+c)}\begin{pmatrix}
					2(s-a)\\
					s-a
				\end{pmatrix} \begin{pmatrix}
					2(s-b)\\
					s-b
				\end{pmatrix} \begin{pmatrix}
					2(s-c)\\
					s-c
				\end{pmatrix} \begin{pmatrix}
					2s\\
					s
				\end{pmatrix}^{-1} \enspace \textnormal{else}.
			\end{cases}
		\end{equation*}
	\end{lemma}
	\begin{theorem}[Coefficients of Legendre polynomials in $C^{\infty}_\Theta$]\label{theorem:FormulaLegCoeffs}
		Let $p_d(t)$ be the orthonormal Legendre polynomial of degree $d$ and ${\mathcal{P}}_{a,b,c}$ as in Property \ref{prop:intLeg}.
		Then the coefficients $b_{k,l}^{(d)}$ of $p_d(t)\Theta(t-s)$ expanded in orthonormal Legendre bases $\{p_k(t)\}_k$ and $\{p_\ell(s)\}$, i.e., $p_d(t)\Theta(t-s) = \sum_{k=0}^\infty \sum_{\ell=0}^\infty b_{k,\ell}^{(d)} p_k(t) p_\ell(s)$ are given by		
		\begin{equation}\label{eq:LegBasisCoeffs}
			b^{(d)}_{k,\ell}==  \frac{1}{\sqrt{2\ell+1}} \left[ \frac{1}{\sqrt{2\ell+3}}\mathcal{P}_{d,k,\ell+1} - \frac{1}{\sqrt{2\ell-1}}\mathcal{P}_{d,k,\ell-1}\right].
		\end{equation}
	\end{theorem}
	\begin{proof}
		By orthonormality of the Legendre polynomials, we can write the coefficients as
		\begin{align*}
			b^{(d)}_{k,\ell} &= \int_{-1}^1  p_d(\tau) p_k(\tau) \left(\int_{-1}^1\Theta(\tau-\rho)  p_\ell(\rho) d\rho\right) d\tau = \int_{-1}^1  p_d(\tau) p_k(\tau) \underbrace{\left(\int_{-1}^\tau p_\ell(\rho) d\rho\right)}_{\text{Apply Lemma \ref{prop:int_shift_Legendre}}} d\tau \\
			&=  \frac{1}{\sqrt{2\ell+1}} \left[ \frac{1}{\sqrt{2\ell+3}}\int_{-1}^1  p_d(\tau) p_k(\tau) p_{\ell+1}(\tau) d\tau - \frac{1}{\sqrt{2\ell-1}}\int_{-1}^1  p_d(\tau) p_k(\tau)  p_{\ell-1}(\tau)  d\tau\right]\\
			&=  \frac{1}{\sqrt{2\ell+1}} \left[ \frac{1}{\sqrt{2\ell+3}}\mathcal{P}_{d,k,\ell+1} - \frac{1}{\sqrt{2\ell-1}}\mathcal{P}_{d,k,\ell-1}\right].
		\end{align*}
	\end{proof}
	An immediate consequence of this formula is that $B^{(d)}$ is a banded matrix with bandwidth $(d+1)$ and therefore $F^{(N)}$ has bandwidth $N$.
	In the sequel, the truncated coefficient matrix $F^{(N)}_M\in\mathbb{C}^{M\times M}$ is used, which is the $M\times M$ leading principal submatrix of $F^{(N)}$.
	Choosing an adequate value of $M$ is outside the scope of this paper and is the subject of ongoing research.
	
	\subsection{A finite basis}\label{sec:finBasis}
	Once $f(t,s)$ is replaced by its truncated coefficient matrix $F^{(N)}_M$, the symbolic operations in the $\star$-algebra are replaced by operations involving matrices.
	Consider $f(t,s),g(t,s)$ and their truncated coefficient matrices $F^{(N)}_M, G^{(N)}_M$, respectively.
	The coefficient matrix of the result of the $\star$-product $q(t,s) = f(t,s)\star g(t,s)$ is then approximately given by the coefficient matrix $Q^{(N)}_M = F^{(N)}_M G^{(N)}_M$, i.e., the usual matrix matrix product.
	This is obtained by plugging in the truncated series in the $\star$-product:
	\begin{align*}
		&q(t,s)=f(t,s)\star g(t,s) = \int_{-1}^1 f(t,\tau) g(\tau,s) d\tau \\
		&\approx \int_{-1}^1 \left(\sum_{k=0}^{M-1} \sum_{\ell=0}^{M-1} f_{k,\ell} p_k(t) p_\ell(\tau)\right) \left(\sum_{k=0}^{M-1} \sum_{\ell=0}^{M-1} g_{k,\ell} p_k(\tau) p_\ell(s)\right) d\tau\\
		&= \begin{bmatrix}
			p_0(t) & \dots & p_M(t)
		\end{bmatrix}F^{(N)}_M  
		\begin{bmatrix}
			\int_{-1}^1 p_0(\tau)p_0(\tau)d\tau & \dots & \int_{-1}^1 p_0(\tau)p_\infty(\tau)d\tau\\
			\int_{-1}^1 p_1(\tau)p_0(\tau) d\tau& \dots & \int_{-1}^1 p_1(\tau)p_\infty(\tau)d\tau\\
			\vdots & & \vdots \\
			\int_{-1}^1 p_{M-1}(\tau) p_0(\tau)d\tau & \dots &\int_{-1}^1  p_M(\tau) p_{M-1}(\tau)d\tau
		\end{bmatrix}
		G^{(N)}_M \begin{bmatrix}
			p_0(s)\\
			\vdots\\
			p_{M-1}(s)
		\end{bmatrix} \\
		&= \begin{bmatrix}
			p_0(t) & \dots & p_{M-1}(t)
		\end{bmatrix} 
		\underbrace{F^{(N)}_M G^{(N)}_M}_{=Q^{(N)}_M} \begin{bmatrix}
			p_0(s)\\
			\vdots\\
			p_{M-1}(s)
		\end{bmatrix}.
	\end{align*}
	Table \ref{table:matrixOps} describes the matrix algebra for the coefficient matrices of functions in $\mathcal{D}$($\mathcal{I}$).
	This matrix algebra is, in some sense, the discretization of the $\star$-algebra.
	The coefficient matrix for the expansion of $\Theta(t-s)$ in the Legendre bases is denoted by $H_M\in\mathcal{C}^{M\times M}$.
	This matrix appears in the coefficient matrix $U^{(N)}_M = H_M (I_M-F_M^{(N)})^{-1}$ which represents the series approximating the solution $u(t,s) = \Theta(t-s)\star (1_{\star}- f)^{\star-1}(t,s)$.
	This expression is used in the next section to approximate the solution $\tilde{u}(t)$.
	\begin{table}[!ht]
		\begin{tabular}{l|l}
			$f(t,s)$ & $F_M$ \\ \hline
			$\star$-operation/elements & matrix operation/elements\\				\hline
			$q(t,s) = f(t,s) \star g(t,s)$             &  $Q_M = F_M G_M$ \\
			$f + g$           & $F_M+G_M$  \\
			$1_{\star} := \delta(t-s)$ &   $I_M$, identity matrix       \\
			$f^{\star-1}(t,s)$                    & $F_M^{-1}$    \\
			$R_\star(f)(t,s):=(1_{\star}- f)^{\star-1}(t,s)$   &  $R(F_M) := (I_M-F_M)^{-1}$\\ \hline \hline
			Solution to ODE \eqref{eq:ODE_bivariate} & Approximate solution\\ \hline
			$u(t,s) = \Theta(t-s)\star (1_{\star}- f)^{\star-1}(t,s)$ & $U_M = H_M (I_M-F_M)^{-1}$
		\end{tabular}
		\caption{Matrix algebra for the coefficient matrices of functions in $\mathcal{D}(\mathcal{I})$ and corresponding operations and elements in the $\star$-algebra.}
		\label{table:matrixOps}
	\end{table}

	\section{Approximation in matrix framework}\label{sec:approx}
	For the solution $u(t,s)$ of \eqref{eq:ODE_bivariate} the truncated coefficient matrix in the Legendre bases is given by $U_M^{(N)} = H_M (I_M-F_M^{(N)})^{-1}$.
	Since the solution of interest $\tilde{u}(t) = u(t,s)\vert_{s=-1}$, its coefficients in Legendre basis $u^{(N)} = \begin{bmatrix}
		u_0 & u_1 & \dots & u_{M-1}
	\end{bmatrix}$ can be computed by solving the linear system $ (I_M-F_M^{(N)}) y =\begin{bmatrix}
		p_0(-1) &  p_1(-1) & \dots & p_{M-1}(-1)
	\end{bmatrix}^\top$ for $y$ and forming the product $u^{(N)}=H_M y$.
	The first $L$ coefficients in $u\in\mathbb{C}^{M}$ can be computed up to high accuracy, which leads to the approximate series $\tilde{u}(t) \approx \sum_{k=0}^{L} u_k p_k(t)=:\hat{u}_L(t)$.\\
	Only the first $L$ coefficients of $u^{(N)}$ are accurate because of the use of a truncated basis.
	The linear system for $y$ involves an $N$-banded matrix $(I_M-F_M^{(N)})$, which is a truncation of infinite matrices. 
	Thanks to the bandedness of $(I_M-F_M^{(N)})$, the first $M-K$ coefficients in $y$ are computed accurately, i.e., as if these coefficients were computed using the infinite matrices.
	Here, $K$ corresponds to the numerical bandedness of the matrix $(I_M-F_M^{(N)})^{-1}$.
	The multiplication with the tridiagonal $H_M$ then results in the first $L = M-K-1$ coefficients in $u$ being accurately computed.\\
	We illustrate this using an example, consider $\tilde{f}(t) = \cos(4t)$, which can be accurately represented in Legendre basis by the series $\sum_{d=0}^{21}f_d p_d(t)$.
	Therefore, its coefficient matrix $F^{(22)}_{101}\in\mathbb{C}^{101\times 101}$ has band size $N=22$, Figure \ref{fig:bandedMatrices} shows the entries of $F^{(22)}_{101}$ that are larger than machine precision $\epsilon_{\textrm{mach}} \approx 2.2204 e-16$.
	The coefficient matrix $U_{101}^{(22)} = H_{101} (I_{101}-F_{101}^{(22)})^{-1}$ is also shown on this figure and we observe that it has a numerical band size of $K+1=29$ in the trailing part of the matrix.
	This suggests that the first $M-K-1 = 71$ of $u^{(22)}$ are accurate.
	Figure \ref{fig:coeffs} verifies that the first 71 computed coefficients $\{u_k\}_k$ by comparing them to the coefficients of the Legendre expansion of the known exact solution $\tilde{u}(t)$ and by comparing the Legendre series $\hat{u}_n(t):=\sum_{k=0}^n u_k p_k(t)$ to the exact solution $\tilde{u}(t)$ in 1000 equispaced nodes in $\mathcal{I}$.\\
	
	\begin{figure}[!ht]
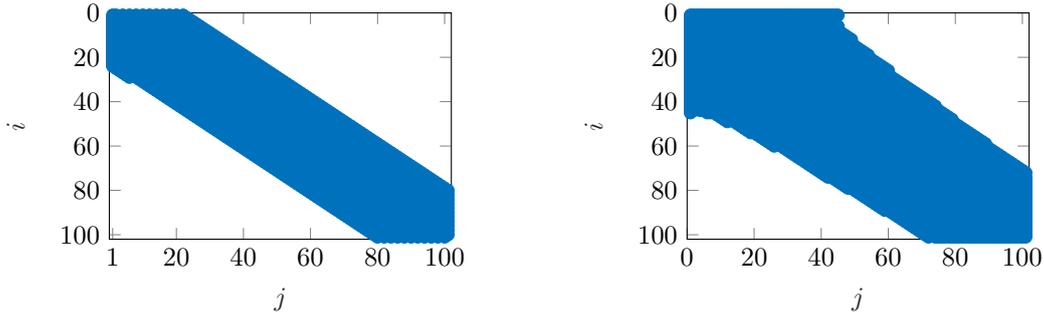

		\centering
		\begin{subfigure}{0.4\textwidth}
			\setlength\figureheight{3cm}
			\setlength\figurewidth{6cm}
			{\input{fig1.tikz}}
		\end{subfigure}
		\hspace{1cm}
		\begin{subfigure}{0.4\textwidth}
			\setlength\figureheight{3cm}
			\setlength\figurewidth{6cm}
			{\input{fig2.tikz}}
		\end{subfigure}
		\caption{Entries in $F^{(22)}_{101}$ (left) for $f(t)=\cos(4t)$ and $U^{(22)}_{101}$ (right) with an amplitude larger than machine precision, $\vert f_{i,j}\vert \leq \epsilon_{\textrm{mach}}$ and $\vert u_{i,j}\vert \leq \epsilon_{\textrm{mach}}$.}
		\label{fig:bandedMatrices}
	\end{figure}

	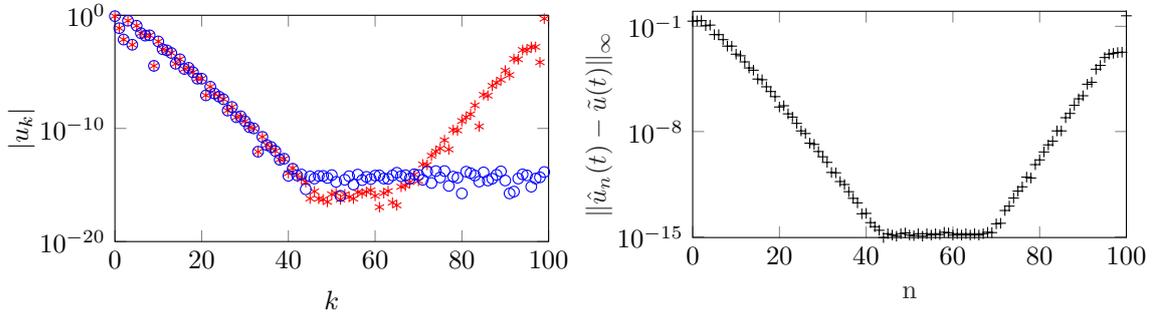
\begin{figure}[!ht]
		\centering
		\hspace{-0.8cm}
		\begin{subfigure}{0.4\textwidth}
			\setlength\figureheight{3cm}
			\setlength\figurewidth{6cm}
			{
%
\begin{tikzpicture}
	
	\begin{axis}[%
		width=0.951\figurewidth,
		height=\figureheight,
		at={(0\figurewidth,0\figureheight)},
		scale only axis,
		xmin=0,
		xmax=1e+02,
		ymode=log,
		ymin=1e-20,
		ymax=1,
		yminorticks=true,
		xlabel = $k$,
		ylabel = $\vert u_k\vert$,
		axis background/.style={fill=white}
		]
		\addplot [color=red, draw=none, mark=asterisk, mark options={solid, red}, forget plot]
		table[row sep=crcr]{%
			0	0.84\\
			1	0.073\\
			2	0.0072\\
			3	0.33\\
			4	0.0025\\
			5	0.12\\
			6	0.026\\
			7	0.016\\
			8	0.017\\
			9	3.4e-05\\
			10	0.0049\\
			11	0.001\\
			12	0.00079\\
			13	0.00047\\
			14	5.2e-05\\
			15	0.00013\\
			16	1.8e-05\\
			17	2.3e-05\\
			18	9.2e-06\\
			19	2.5e-06\\
			20	2.5e-06\\
			21	8.4e-08\\
			22	4.9e-07\\
			23	1.2e-07\\
			24	6.5e-08\\
			25	3.8e-08\\
			26	3.8e-09\\
			27	7.9e-09\\
			28	9.8e-10\\
			29	1.2e-09\\
			30	4.3e-10\\
			31	1.3e-10\\
			32	1e-10\\
			33	8.8e-13\\
			34	1.8e-11\\
			35	3.4e-12\\
			36	2.3e-12\\
			37	1e-12\\
			38	1.8e-13\\
			39	2e-13\\
			40	1.1e-14\\
			41	3.1e-14\\
			42	7.9e-15\\
			43	3.3e-15\\
			44	1.8e-15\\
		};
		\addplot [color=red, draw=none, mark=asterisk, mark options={solid, red}, forget plot]
		table[row sep=crcr]{%
			45	6.3e-17\\
			46	2.7e-16\\
			47	6e-17\\
			48	5.3e-17\\
			49	3.2e-17\\
			50	1.5e-16\\
			51	1.9e-16\\
			52	5.3e-17\\
			53	1.5e-16\\
			54	8.7e-17\\
			55	6.2e-17\\
			56	2.3e-16\\
			57	1.8e-16\\
			58	2.5e-16\\
			59	3.4e-16\\
			60	1.1e-16\\
			61	1.1e-17\\
			62	1.9e-16\\
			63	2.9e-16\\
			64	2.8e-17\\
			65	1.6e-17\\
			66	7e-16\\
		};
		\addplot [color=red, draw=none, mark=asterisk, mark options={solid, red}, forget plot]
		table[row sep=crcr]{%
			67	8.5e-16\\
			68	1.8e-15\\
			69	8.3e-15\\
			70	2.3e-15\\
			71	6.5e-14\\
			72	5.9e-14\\
			73	3.9e-13\\
			74	9.1e-13\\
			75	1.7e-12\\
			76	9.2e-12\\
			77	1.4e-12\\
			78	7.3e-11\\
			79	6.7e-11\\
			80	4.5e-10\\
			81	1e-09\\
			82	1.8e-09\\
			83	1.1e-08\\
			84	1.5e-10\\
			85	9.3e-08\\
			86	7.9e-08\\
			87	5.8e-07\\
			88	1.1e-06\\
			89	2e-06\\
			90	1.3e-05\\
			91	5.2e-06\\
			92	0.00014\\
			93	0.00013\\
			94	0.00086\\
			95	0.00084\\
			96	0.0016\\
			97	0.0016\\
			98	7.1e-05\\
			99	0.5\\
		};
		\addplot [color=blue, draw=none, mark=o, mark options={solid, blue}, forget plot]
			table[row sep=crcr]{%
				0	0.84\\
				1	0.073\\
				2	0.0072\\
				3	0.33\\
				4	0.0025\\
				5	0.12\\
				6	0.026\\
				7	0.016\\
				8	0.017\\
				9	3.4e-05\\
				10	0.0049\\
				11	0.001\\
				12	0.00079\\
				13	0.00047\\
				14	5.2e-05\\
				15	0.00013\\
				16	1.8e-05\\
				17	2.3e-05\\
				18	9.2e-06\\
				19	2.5e-06\\
				20	2.5e-06\\
				21	8.4e-08\\
				22	4.9e-07\\
				23	1.2e-07\\
				24	6.5e-08\\
				25	3.8e-08\\
				26	3.8e-09\\
				27	7.9e-09\\
				28	9.8e-10\\
				29	1.2e-09\\
				30	4.3e-10\\
				31	1.3e-10\\
				32	1e-10\\
				33	8.8e-13\\
				34	1.8e-11\\
				35	3.4e-12\\
				36	2.3e-12\\
				37	1e-12\\
				38	1.7e-13\\
				39	2.1e-13\\
				40	6.5e-15\\
				41	2.7e-14\\
				42	7.1e-15\\
				43	8.5e-15\\
				44	4.2e-16\\
				45	5.7e-15\\
				46	3.7e-15\\
				47	6.3e-15\\
				48	6.4e-15\\
				49	4.2e-15\\
				50	7.5e-15\\
				51	1.9e-15\\
				52	1.1e-16\\
				53	2.9e-15\\
				54	6e-15\\
				55	1.1e-15\\
				56	8.3e-15\\
				57	2.7e-15\\
				58	5.2e-15\\
				59	3.4e-15\\
				60	6.8e-15\\
				61	7e-15\\
				62	3.7e-15\\
				63	3.5e-15\\
				64	7e-15\\
				65	1.2e-14\\
				66	5.9e-15\\
				67	7.3e-15\\
				68	3.6e-15\\
				69	8.9e-15\\
				70	7.6e-15\\
				71	2.9e-15\\
				72	1.5e-14\\
				73	1.5e-15\\
				74	1.7e-14\\
				75	4.6e-15\\
				76	1.7e-14\\
				77	8.3e-16\\
				78	5.5e-15\\
				79	4.4e-15\\
				80	1.8e-16\\
				81	1.5e-14\\
				82	1.1e-14\\
				83	6.9e-15\\
				84	2e-15\\
				85	1.3e-14\\
				86	2.5e-15\\
				87	3.6e-15\\
				88	7.9e-15\\
				89	1.6e-14\\
				90	2.5e-15\\
				91	1.8e-16\\
				92	2.6e-16\\
				93	3.7e-15\\
				94	8.7e-15\\
				95	7.1e-15\\
				96	7.1e-16\\
				97	3.7e-15\\
				98	4.8e-15\\
				99	1.4e-14\\
			};
	\end{axis}
\end{tikzpicture}%
}
		\end{subfigure}
		\hspace{1cm}
		\begin{subfigure}{0.4\textwidth}
			\setlength\figureheight{3cm}
			\setlength\figurewidth{6cm}
			{
%
\begin{tikzpicture}

\begin{axis}[%
width=0.951\figurewidth,
height=\figureheight,
at={(0\figurewidth,0\figureheight)},
scale only axis,
xmin=0,
xmax=1e+02,
xlabel style={font=\color{white!15!black}},
xlabel={n},
ymode=log,
ymin=8.9e-16,
ymax=1,
yminorticks=true,
ylabel style={font=\color{white!15!black}},
ylabel={$\Vert \hat{u}_n(t)-\tilde{u}(t) \Vert_\infty$},
axis background/.style={fill=white}
]
\addplot [color=black, draw=none, mark=+, mark options={solid, black}, forget plot]
  table[row sep=crcr]{%
0	0.22\\
1	0.23\\
2	0.24\\
3	0.11\\
4	0.12\\
5	0.028\\
6	0.028\\
7	0.013\\
8	0.0048\\
9	0.0048\\
10	0.0014\\
11	0.0011\\
12	0.00043\\
13	0.00017\\
14	0.00012\\
15	3.1e-05\\
16	2.8e-05\\
17	9.5e-06\\
18	4.6e-06\\
19	2.1e-06\\
20	4.3e-07\\
21	5.2e-07\\
22	1.5e-07\\
23	9.2e-08\\
24	3.5e-08\\
25	1.1e-08\\
26	7.4e-09\\
27	1.7e-09\\
28	1.4e-09\\
29	4.7e-10\\
30	2.1e-10\\
31	8.9e-11\\
32	1.9e-11\\
33	1.8e-11\\
34	4.7e-12\\
35	3e-12\\
36	1e-12\\
37	3.6e-13\\
38	1.8e-13\\
39	3.4e-14\\
40	3.6e-14\\
41	8.7e-15\\
42	4.7e-15\\
43	2.7e-15\\
44	8.9e-16\\
45	1.6e-15\\
46	1.3e-15\\
47	1.1e-15\\
48	1.6e-15\\
49	1.8e-15\\
50	1.3e-15\\
51	1.2e-15\\
52	1.6e-15\\
53	1.1e-15\\
54	1.6e-15\\
55	1.3e-15\\
56	1.6e-15\\
57	1.4e-15\\
58	2e-15\\
59	1.8e-15\\
60	1.3e-15\\
61	1.6e-15\\
62	1.3e-15\\
63	1.6e-15\\
64	1.3e-15\\
65	1.6e-15\\
66	1.3e-15\\
67	1.6e-15\\
68	2e-15\\
69	1.9e-15\\
70	7.3e-15\\
71	7.9e-15\\
72	5.6e-14\\
73	1.1e-13\\
74	3.8e-13\\
75	1.1e-12\\
76	2.1e-12\\
77	9.3e-12\\
78	8.1e-12\\
79	6.3e-11\\
80	1.3e-10\\
81	4.4e-10\\
82	1.3e-09\\
83	2.3e-09\\
84	1.1e-08\\
85	1.1e-08\\
86	8.3e-08\\
87	1.6e-07\\
88	5.6e-07\\
89	1.5e-06\\
90	2.4e-06\\
91	1.4e-05\\
92	1.9e-05\\
93	0.00014\\
94	0.00025\\
95	0.00084\\
96	0.0014\\
97	0.0016\\
98	0.0018\\
99	0.0019\\
1e+02	0.5\\
};
\end{axis}
\end{tikzpicture}%
}
		\end{subfigure}
		\caption{Left: Magnitude of coefficients representing the solution to the ODE with $\tilde{f}(t) = \cos(4t)$. The coefficients of the Legendre series of the exact solution $\tilde{u}(t)$ are shown as ${\color{blue}\circ}$ and the coefficients $u = H_M(I_M-F_M)^{-1}y$ obtained by the method described in this paper as ${\color{red}\ast}$.
			Right: error of the series $\hat{u}_n(t) = \sum_{k=0}^{n} u_k p_k(t)$ compared to $\tilde{u(t)}$ measured in infinity norm for increasing $n$.}
		\label{fig:coeffs}
	\end{figure}

	Consider now a more oscillatory function $\tilde{f}(t) = -2\pi \imath( 0.1 + \cos(6\pi(t+1))  + \cos(12\pi(t+1)))$ for which we repeat the experiment.
	We choose $M = 601$ and the functions is represented by $\tilde{f}(t) \approx \sum_{d=0}^{74} f_d p_d(t)$.
	Numerically we determine that the band of $(I_{601}-F^{75}_{601})^{-1}$ is $K = 196$, which suggests that $M-K-1 = 404$ coefficients are computed accurately.
	This is verified in Figure \ref{fig:coeffs_moreDifficult}, where the computed coefficients are compared with the exact Legendre coefficients and the accuracy of $\hat{u}_n(t)$ is compared to $\tilde{u}(t)$ in the infinity norm.
	\begin{figure}[!ht]
		\centering
		\hspace{-0.8cm}
		\begin{subfigure}{0.4\textwidth}
			\setlength\figureheight{3cm}
			\setlength\figurewidth{6cm}
			{\input{fig5.tikz}}
		\end{subfigure}
		\hspace{1cm}
		\begin{subfigure}{0.4\textwidth}
			\setlength\figureheight{3cm}
			\setlength\figurewidth{6cm}
			{
%
\begin{tikzpicture}

\begin{axis}[%
width=0.951\figurewidth,
height=\figureheight,
at={(0\figurewidth,0\figureheight)},
scale only axis,
xmin=0,
xmax=6e+02,
xlabel style={font=\color{white!15!black}},
xlabel={n},
ymode=log,
ymin=1e-15,
ymax=1,
yminorticks=true,
ylabel style={font=\color{white!15!black}},
ylabel={$\Vert \hat{u}_n(t)-\tilde{u}(t) \Vert_\infty$},
axis background/.style={fill=white}
]
\addplot [color=black, draw=none, mark=+, mark options={solid, black}, forget plot]
  table[row sep=crcr]{%
0	0.83\\
5	0.47\\
10	0.46\\
15	0.39\\
20	0.2\\
25	0.19\\
30	0.19\\
35	0.24\\
40	0.066\\
45	0.037\\
50	0.041\\
55	0.03\\
60	0.011\\
65	0.009\\
70	0.012\\
75	0.0089\\
80	0.0013\\
85	0.0016\\
90	0.0018\\
95	0.0011\\
1e+02	0.00024\\
1e+02	0.00024\\
1.1e+02	0.0004\\
1.2e+02	0.00015\\
1.2e+02	4.5e-05\\
1.2e+02	4.4e-05\\
1.3e+02	5.2e-05\\
1.4e+02	1.4e-05\\
1.4e+02	5.8e-06\\
1.4e+02	9.6e-06\\
1.5e+02	7.1e-06\\
1.6e+02	1.9e-06\\
1.6e+02	9.5e-07\\
1.6e+02	1.5e-06\\
1.7e+02	8e-07\\
1.8e+02	1.2e-07\\
1.8e+02	1.5e-07\\
1.8e+02	1.9e-07\\
1.9e+02	8.1e-08\\
2e+02	1.8e-08\\
2e+02	2.7e-08\\
2e+02	2.2e-08\\
2.1e+02	5.7e-09\\
2.2e+02	2.6e-09\\
2.2e+02	3.4e-09\\
2.2e+02	2.2e-09\\
2.3e+02	6.2e-10\\
2.4e+02	4e-10\\
2.4e+02	4.1e-10\\
2.4e+02	1.8e-10\\
2.5e+02	6.7e-11\\
2.6e+02	4.8e-11\\
2.6e+02	4.2e-11\\
2.6e+02	1.8e-11\\
2.7e+02	7e-12\\
2.8e+02	5.9e-12\\
2.8e+02	3.8e-12\\
2.8e+02	1.7e-12\\
2.9e+02	6.8e-13\\
3e+02	6.2e-13\\
3e+02	3.6e-13\\
3e+02	1.4e-13\\
3.1e+02	7.6e-14\\
3.2e+02	6.3e-14\\
3.2e+02	3.6e-14\\
3.2e+02	1.2e-14\\
3.3e+02	7e-15\\
3.4e+02	9.8e-15\\
3.4e+02	7.9e-15\\
3.4e+02	7.6e-15\\
3.5e+02	5.6e-15\\
3.6e+02	6.5e-15\\
3.6e+02	6.7e-15\\
3.6e+02	5.9e-15\\
3.7e+02	8.6e-15\\
3.8e+02	5.9e-15\\
3.8e+02	6.3e-15\\
3.8e+02	1.2e-14\\
3.9e+02	2.2e-14\\
4e+02	3.5e-14\\
4e+02	4e-14\\
4e+02	2e-13\\
4.1e+02	5.9e-13\\
4.2e+02	1e-12\\
4.2e+02	7.7e-13\\
4.2e+02	7e-13\\
4.3e+02	3.3e-12\\
4.4e+02	4.8e-12\\
4.4e+02	1.1e-11\\
4.4e+02	4.6e-11\\
4.5e+02	1.2e-10\\
4.6e+02	1.4e-10\\
4.6e+02	4.4e-11\\
4.6e+02	1.8e-10\\
4.7e+02	5.9e-10\\
4.8e+02	3.9e-10\\
4.8e+02	2.7e-09\\
4.8e+02	1.2e-08\\
4.9e+02	2.5e-08\\
5e+02	7.8e-09\\
5e+02	5.6e-09\\
5e+02	5.8e-08\\
5.1e+02	6.3e-08\\
5.2e+02	9.3e-08\\
5.2e+02	8.9e-07\\
5.2e+02	3.6e-06\\
5.3e+02	3.4e-06\\
5.4e+02	3e-06\\
5.4e+02	4.6e-06\\
5.4e+02	6.8e-06\\
5.5e+02	5e-06\\
5.6e+02	2.5e-05\\
5.6e+02	0.00032\\
5.6e+02	0.001\\
5.7e+02	0.00072\\
5.8e+02	0.00056\\
5.8e+02	0.00051\\
5.8e+02	0.00044\\
5.9e+02	0.0003\\
6e+02	0.0016\\
6e+02	0.5\\
};
\end{axis}
\end{tikzpicture}%
}
		\end{subfigure}
		\caption{Left: Magnitude of coefficients representing the solution to the ODE with $\tilde{f}(t) = -2\pi \imath( 0.1 + \cos(6\pi(t+1))  + \cos(12\pi(t+1)))$. The coefficients of the Legendre series of the exact solution $\tilde{u}(t)$ are shown as ${\color{blue}\circ}$ and the coefficients $u^{(N)} = H_M(I_M-F_M^{(N)})^{-1}y$ obtained by the method described in this paper as ${\color{red}\ast}$.
			Right: error of the series $\hat{u}_n(t) = \sum_{k=0}^{n} u_k p_k(t)$ compared to $\tilde{u(t)}$ measured in infinity norm for increasing $n$.}
		\label{fig:coeffs_moreDifficult}
	\end{figure}
	
	These numerical experiments illustrate that the proposed method is capable of solving scalar ODEs up to high accuracy.
	This method can be generalized to the matrix ODE, where similar numerical behavior has been observed, however treating this case is outside the scope of this report.
	In order to develop a numerical algorithm that can compete with the state-of-the-art methods, the coefficient computation must be performed very efficiently, an a priori estimate of the required size of basis $M$ and a procedure to automatically truncate the series $\sum_{k=0}^{n} u_k p_k(t)$ at an appropriate value of $n$.

	\vspace{\baselineskip}


\begin{thebibliography}{1}
		
		
		\bibitem{CiPoRZVB22} 
		S., Cipolla, S., Pozza, M., Redivo-Zaglia, and N., Van Buggenhout, A Lanczos-type procedure for tensors. Numer Algor (2022). Published online: \url{https://doi.org/10.1007/s11075-022-01351-6}.
		
		\bibitem{DrHaTr14}%
		T. A., Driscoll, N., Hale, and L. N., Trefethen, Chebfun guide, (Pafnuty Publications, Oxford, 2014). \url{www.chebfun.org/docs/guide/}.
		
		\bibitem{HaSp98}%
		S., Hafner, and H. W., Spiess, Advanced solid-state NMR spectroscopy of strongly dipolar coupled spins under fast magic angle spinning.
		Concepts Magn. Reson., \textbf{10}, 99-128 (1998). \url{https://doi.org/10.1002/(SICI)1099-0534(1998)10:2<99::AID-CMR3>3.0.CO;2-Q}.	
		
		
		\bibitem{GiJeZe88}%
		J., Gillis, J., Jedwab, and D., Zeilberger, A Combinatorial Interpretation of the Integral of the Product of Legendre Polynomials.
		SIAM J. Math. Anal. \textbf{19:6}, 1455-1461 (1988). \url{https://doi.org/10.1137/0519109}.
		
		\bibitem{GiPo20}%
		P-L., Giscard, and S. Pozza, Lanczos-Like Algorithm for the Time-Ordered Exponential: The $\ast$-Inverse Problem.
		Linear Algebra Appl \textbf{65}, 807–827 (2020). \url{https://doi.org/10.21136/AM.2020.0342-19}.
		
		\bibitem{GiPo21}%
		P-L., Giscard, and S. Pozza, Tridiagonalization of systems of coupled linear differential equations with variable coefficients by a {L}anczos-like method.
		Appl Math \textbf{624}, 153–173 (2021). \url{https://doi.org/10.1016/j.laa.2021.04.011}.
		
		\bibitem{GiPo22}%
		P-L., Giscard, and S. Pozza, A Lanczos-like method for non-autonomous linear ordinary differential equations.
		Boll Unione Mat Ital (2022). Published online: \url{https://doi.org/10.1007/s40574-022-00328-6}.
		
		
		\bibitem{Tr13}%
		L. N., Trefethen, Approximation Theory and Approximation Practice. (SIAM, Philadelphia, PA, 2013).
		
		\bibitem{WaXi12}
		H., Wang, and S., Xiang, On the convergence rates of Legendre approximation.
		Math. Comp. \textbf{81}, 861-877 (2012). \url{https://doi.org/10.1090/S0025-5718-2011-02549-4}.
		
	\end{thebibliography}
\end{document}